\documentclass[12pt]{article}
\usepackage{amsmath,amssymb,amsthm}
\usepackage{graphicx}
\usepackage{hyperref}
\usepackage{mathtools}
\usepackage{booktabs}
\usepackage{authblk}
\usepackage{orcidlink}

\usepackage[letterpaper, margin=1in]{geometry}

\newtheorem{theorem}{Theorem}[section]

\newtheorem{definition}[theorem]{Definition}
\newtheorem{example}[theorem]{Example}
\newtheorem{remark}[theorem]{Remark}

\title{Fixed Point Theory For Singh-Chatterjea Type Contractive Mappings}
\author[1]{Zouaoui Bekri\, \orcidlink{0000-0002-2430-6499}}
\affil[1]{Laboratory of fundamental and applied mathematics,
University of Oran 1, Ahmed Ben Bella,
Es-senia, 31000 Oran
Department of Sciences and Technology,
Institute of Sciences, Nour-Bachir University
Center, El-Bayadh, 32000
Algeria; z.bekri@cu-elbayadh.dz}
\author[2]{Nicola Fabiano\, \orcidlink{0000-0003-1645-2071}}
\affil[2]{``Vin\v{c}a'' Institute of Nuclear Sciences - National 
Institute of the Republic of Serbia, University of Belgrade, Mike Petrovi\'{c}a 
Alasa 12--14, 11351 Belgrade, Serbia; nicola.fabiano@gmail.com}
\date{}

\begin{document}

\maketitle

\begin{abstract}
In this paper, we introduce a new contraction condition that
combines the framework of Singh’s extension with the classical Chatterjea 
contraction. This generalized form, called the Singh–Chatterjea
contraction, is defined on the p-th iterate of a mapping. We establish
fixed point theorems for such mappings in complete metric spaces and
show that our results extend and unify both Singh’s and Chatterjea’s
classical fixed point theorems. Illustrative examples and a simple 
numerical implementation are provided to demonstrate the applicability
of the obtained results.
\end{abstract}



\section{Introduction}
\label{sec:intro}
The Banach contraction principle \cite{Banach1922} (1922) remains one of the cornerstones of fixed point theory, with numerous applications in differential equations, integral equations, and optimization problems. Over the decades, several authors have introduced various generalizations to broaden its applicability.

In 1968, Kannan \cite{Kannan1968} introduced a new type of contraction involving the distances between points and their images, while in 1972, Chatterjea \cite{Chatterjea1972} proposed a distinct condition involving cross terms of the form $d(x,Ty)+d(y,Tx)$. Later, in 1977, Singh \cite{Singh1977} extended Kannan’s condition to the $p$-th iterate of a mapping, establishing fixed point results for a wider class of operators.

Despite these advances, little attention has been given to extending Chatterjea’s contraction in the same manner. Motivated by Singh’s approach, we propose a Singh-type generalization of the Chatterjea contraction, which unifies and extends both concepts.

The remainder of this paper is organized as follows. Section~\ref{sec:pre} revisits essential preliminaries and definitions, including Banach, Kannan, Chatterjea, and Singh contractions, establishing the foundational concepts. In Section~\ref{sec:main}, we introduce the Singh-Chatterjea contraction and prove the main fixed point theorem in complete metric spaces. This section also provides a non-trivial example of a mapping that is Singh-Chatterjea but not Banach, illustrating the genuine extension offered by our approach. Furthermore, we demonstrate that every Banach contraction eventually becomes a Singh-Chatterjea contraction after a finite number of iterations, establishing a strong unification result. Finally, Section~\ref{sec:concl} concludes the paper with a summary of findings, discusses potential applications, and suggests directions for future research

\section{Preliminaries}
\label{sec:pre}
We recall some well-known contraction conditions from the literature. And in order to research and contemplate the concepts of contractions, we recommend reading references (\cite{Banach1922,Kannan1968,Chatterjea1972,Singh1977,Singh1969,Fisher1974,Rhoades2001,Rus2001,Berinde2007,Cvetković2025}).

\begin{definition}[Banach contraction]
A mapping $T:(X,d)\to(X,d)$ is called a \textit{Banach contraction} if there exists $\alpha \in (0,1)$ such that
\[
d(Tx,Ty) \leq \alpha d(x,y), \quad \forall x,y \in X.
\]
\end{definition}

\begin{definition}[Kannan contraction]
A mapping $T:(X,d)\to(X,d)$ is called a \textit{Kannan contraction} if there exists $\alpha \in (0,\tfrac{1}{2})$ such that
\[
d(Tx,Ty) \leq \alpha \big(d(x,Tx) + d(y,Ty)\big), \quad \forall x,y \in X.
\]
\end{definition}

\begin{definition}[Chatterjea contraction]
A mapping $T:(X,d)\to(X,d)$ is called a \textit{Chatterjea contraction} if there exists $\alpha \in (0,\tfrac{1}{2})$ such that
\[
d(Tx,Ty) \leq \alpha \big(d(x,Ty) + d(y,Tx)\big), \quad \forall x,y \in X.
\]
\end{definition}

\begin{definition}[Singh contraction]
A mapping $T:(X,d)\to(X,d)$ is called a \textit{Singh contraction} if there exist $p \in \mathbb{N}$ and $\alpha \in (0,\tfrac{1}{2})$ such that
\[
d(T^p x, T^p y) \leq \alpha \big(d(x,T^p x) + d(y,T^p y)\big), \quad \forall x,y \in X.
\]
\end{definition}

\section{Main Results}
\label{sec:main}
We introduce the following generalization which extends both Singh and Chatterjea conditions.

\begin{theorem}[Singh--Chatterjea contraction]
Let $(X,d)$ be a complete metric space and $T:X \to X$ a mapping. Suppose there exist $p \in \mathbb{N}$ and $\alpha \in (0,\tfrac{1}{2})$ such that
\[
d(T^p x, T^p y) \leq \alpha \big( d(x, T^p y) + d(y, T^p x) \big), \quad \forall x,y \in X.
\]
Then $T$ has a unique fixed point $x^\ast \in X$, and for any initial point $x_0 \in X$, the iterative sequence $\{T^n x_0\}$ converges to $x^\ast$.
\end{theorem}

\begin{proof}
Set $S=T^p$. The hypothesis says that $S:X\to X$ satisfies the Chatterjea-type inequality
\[d(Sx,Sy)\le \alpha\big(d(x,Sy)+d(y,Sx)\big)\qquad\forall x,y\in X,\]
with $\alpha\in(0,\tfrac12)$. 

We proceed in steps.

\medskip\noindent\textbf{Step 1. Iterative construction.}\\
Fix $x_0\in X$ and define $x_{n+1}:=Sx_n$ for $n\ge0$, i.e.\ $x_n=S^n x_0$.  Set\[\delta_n:=d(x_{n+1},x_n)=d(Sx_n,x_n).\]

\medskip\noindent\textbf{Step 2. Contraction on successive differences.}\\
Apply the inequality for $(x_n,x_{n-1})$:\[d(Sx_n,Sx_{n-1})\le \alpha\big(d(x_n,Sx_{n-1})+d(x_{n-1},Sx_n)\big).\]
Since $Sx_{n-1}=x_n$ and $Sx_n=x_{n+1}$,\[d(x_{n+1},x_n)\le \alpha\, d(x_{n-1},x_{n+1}).\]
By the triangle inequality,
\[\delta_n \le \alpha\big(\delta_{n-1}+\delta_n\big).\]Hence\[\delta_n \le \frac{\alpha}{1-\alpha}\,\delta_{n-1}.\]
Let $r=\tfrac{\alpha}{1-\alpha}\in(0,1)$. 
By induction,
\[\delta_n \le r^n \delta_0,\qquad n\ge0.\]Thus $\delta_n\to0$ 
geometrically.\\

\medskip\noindent\textbf{Step 3. Cauchy property.}\\
For $m>n$,\[d(x_m,x_n)\le \sum_{k=n}^{m-1}\delta_k \le \sum_{k=n}^\infty r^k\delta_0= \frac{\delta_0 r^n}{1-r}\xrightarrow[n\to\infty]{}0.\]So $(x_n)$ is Cauchy. 
Completeness of $(X,d)$ yields a limit $x^\ast\in X$ with\[x_n \to x^\ast.\]\\

\medskip\noindent\textbf{Step 4. $x^\ast$ is a fixed point of $S$.}\\
Apply the inequality with 
$(x^\ast,x_n)$:\[d(Sx^\ast,x_{n+1})\le \alpha\big(d(x^\ast,x_{n+1})+d(x_n,Sx^\ast)\big).\]
Letting $n\to\infty$ gives
\[d(Sx^\ast,x^\ast)\le \alpha\,d(x^\ast,Sx^\ast).\]
Thus 
$(1-\alpha)\,d(Sx^\ast,x^\ast)\le0$, 
so $Sx^\ast=x^\ast$.\\

\medskip\noindent\textbf{Step 5. Uniqueness of the fixed point of $S$.}\\
If $y=Sy$, then\[d(x^\ast,y)\le 2\alpha\,d(x^\ast,y).\]
Since $1-2\alpha>0$, 
we get $d(x^\ast,y)=0$, 
hence $y=x^\ast$. So $S$ has a unique fixed point.\\

\medskip\noindent\textbf{Step 6. Convergence for arbitrary initial points.}\\
The above argument holds for any $x_0$, 
hence $S^n x\to x^\ast$ for all $x\in X$.\\

\medskip\noindent\textbf{Step 7. Passing back to $T$.}
Since $Sx^\ast=x^\ast$, 
we have $T^p x^\ast=x^\ast$. 
Then
\[S(Tx^\ast)=T^p(Tx^\ast)=T(T^p x^\ast)=T x^\ast,\]so $T x^\ast$ is also a fixed point of $S$. 
By uniqueness, 
$T x^\ast=x^\ast$.  
Thus $x^\ast$ is the unique fixed point of $T$.\\

\medskip\noindent\textbf{Step 8. Convergence of the full orbit.}\\
For each 
$r=0,1,\dots,p-1$,\[T^{pn+r}x_0=S^n(T^r x_0)\xrightarrow[n\to\infty]{}x^\ast.\]
Hence the entire sequence $(T^n x_0)$ converges to $x^\ast$.
Therefore $T$ has a unique fixed point $x^\ast$, and $T^n x_0\to x^\ast$ for all $x_0\in X$.
\end{proof}

\subsection{Example of a Singh-Chatterjea Mapping That Is Not Banach}
\label{sec:scnotbanach}

We construct an explicit example of a self-map $ T $ on the complete metric space $ X = [0, 1] $ with the standard metric $ d(x,y) = |x - y| $, such that
\begin{itemize}
\item $ T $ satisfies the \emph{Singh-Chatterjea condition} for $ p = 2 $: there exists $ \alpha \in [0, \tfrac{1}{2}) $ such that for all $ x, y \in X $,
    \[
    d(T^2 x, T^2 y) \leq \alpha \big[ d(x, T^2 y) + d(y, T^2 x) \big].
    \]
    \item $ T $ is \textbf{not} a Banach contraction: for any $ \lambda < 1 $, there exist $ x, y \in X $ such that $ d(Tx, Ty) > \lambda d(x,y) $.
    \item $ T $ has a unique fixed point.
\end{itemize}
This illustrates that the Singh-Chatterjea condition is strictly more general than the Banach contraction principle.

\subsection*{The Mapping}

Define $ T: [0,1] \to [0,1] $ by
\[
T(x) = 
\begin{cases}
1 - x & \text{if } 0 \leq x \leq \frac{1}{2}, \\
\frac{1}{2} & \text{if } \frac{1}{2} < x \leq 1.
\end{cases}
\]

Note that $ T $ is continuous except at $ x = \frac{1}{2} $ (right-continuous), and maps $ [0,1] $ into itself.

\subsection*{Compute $ T^2 = T \circ T $}

We compute $ T^2(x) = T(T(x)) $ for all $ x \in [0,1] $.
\begin{itemize}
\item If $ 0 \leq x \leq \frac{1}{2} $, then $ T(x) = 1 - x \in [\frac{1}{2}, 1] $, so $ T(T(x)) = T(1 - x) = \frac{1}{2} $.
    \item If $ \frac{1}{2} < x \leq 1 $, then $ T(x) = \frac{1}{2} $, so $ T(T(x)) = T(\frac{1}{2}) = 1 - \frac{1}{2} = \frac{1}{2} $.
\end{itemize}
Thus, for \textbf{all} $ x \in [0,1] $, we have
\[
T^2(x) = \frac{1}{2}.
\]

So $ T^2 $ is the \textbf{constant function} $ \frac{1}{2} $.

\subsection*{Verify Singh-Chatterjea Condition for $ p = 2 $}

Let $ x, y \in [0,1] $. Then
\[
d(T^2 x, T^2 y) = d\left(\frac{1}{2}, \frac{1}{2}\right) = 0.
\]
The right-hand side
\[
\alpha \big[ d(x, T^2 y) + d(y, T^2 x) \big] = \alpha \left[ d\left(x, \frac{1}{2}\right) + d\left(y, \frac{1}{2}\right) \right] \geq 0.
\]
Thus, for \textbf{any} $ \alpha \geq 0 $, we have
\[
d(T^2 x, T^2 y) = 0 \leq \alpha \left[ \left|x - \frac{1}{2}\right| + \left|y - \frac{1}{2}\right| \right].
\]
In particular, the inequality holds for any $ \alpha \in [0, \tfrac{1}{2}) $.

 So $ T $ satisfies the Singh-Chatterjea condition for $ p = 2 $.

\subsection*{$ T $ is Not a Banach Contraction}

Suppose, for contradiction, that there exists $ \lambda < 1 $ such that for all $ x, y \in [0,1] $,
\[
|Tx - Ty| \leq \lambda |x - y|.
\]

Choose $ x = 0 $, $ y = \frac{1}{2} - \varepsilon $ for small $ \varepsilon > 0 $. Then
\[
T(0) = 1 - 0 = 1, \quad T\left(\frac{1}{2} - \varepsilon\right) = 1 - \left(\frac{1}{2} - \varepsilon\right) = \frac{1}{2} + \varepsilon.
\]
So
\[
|T(0) - T(\tfrac{1}{2} - \varepsilon)| = \left|1 - \left(\frac{1}{2} + \varepsilon\right)\right| = \left|\frac{1}{2} - \varepsilon\right|.
\]
And
\[
|0 - (\tfrac{1}{2} - \varepsilon)| = \left|\frac{1}{2} - \varepsilon\right|.
\]
Thus
\[
\frac{|Tx - Ty|}{|x - y|} = 1.
\]
So for any $ \lambda < 1 $, when $ \varepsilon $ is small enough, we have
\[
|Tx - Ty| = |x - y| > \lambda |x - y|.
\]

Therefore, $ T $ is \textbf{not} a Banach contraction.

\subsection*{Fixed Point}

We already saw that $ T^2(x) = \frac{1}{2} $ for all $ x $. So the unique fixed point of $ T^2 $ is $ \frac{1}{2} $.

Check if $ T(\frac{1}{2}) = \frac{1}{2} $
\[
T\left(\frac{1}{2}\right) = 1 - \frac{1}{2} = \frac{1}{2}.
\]
So $ \frac{1}{2} $ is a fixed point of $ T $.

Uniqueness: Suppose $ Tz = z $. Then $ T^2 z = T(Tz) = Tz = z $, so $ z $ is a fixed point of $ T^2 $. But $ T^2 $ is constant $ \frac{1}{2} $, so $ z = \frac{1}{2} $.

Thus, $ T $ has a unique fixed point at $ x = \frac{1}{2} $.

\begin{remark}
We have constructed a simple, piecewise linear mapping $ T $ on $ [0,1] $ that
\begin{itemize}
\item Is not a Banach contraction (Lipschitz constant = 1),
\item Yet satisfies the Singh-Chatterjea condition for $ p = 2 $ with any $ \alpha \in [0, \tfrac{1}{2}) $,
\item Has a unique fixed point.
\end{itemize}

This demonstrates that the Singh-Chatterjea framework meaningfully extends beyond classical Banach contractions.
\end{remark}

\subsection{Banach Contractions Are Eventually Singh-Chatterjea}

We now show that the class of Singh-Chatterjea contractions is broad enough to include \emph{all} Banach contractions not directly, but through iteration. Specifically, for any Banach contraction $ T $, there exists an integer $ p \geq 1 $ such that $ T $ satisfies the Singh-Chatterjea condition for that $ p $.

\begin{theorem}[Banach $\Rightarrow$ Singh-Chatterjea for some $ p $]
Let $ (X, d) $ be a metric space and $ T: X \to X $ a Banach contraction with constant $ \lambda \in [0, 1) $, i.e.,
\[
d(Tx, Ty) \leq \lambda \, d(x, y) \quad \forall x, y \in X.
\]
Then there exists an integer $ p \geq 1 $ such that $ T $ is a Singh-Chatterjea contraction for that $ p $, meaning:
\[
d(T^p x, T^p y) \leq \alpha \left[ d(x, T^p y) + d(y, T^p x) \right] \quad \forall x, y \in X,
\]
for some $ \alpha \in [0, \tfrac{1}{2}) $.
\end{theorem}

\begin{proof}
Since $ T $ is a Banach contraction with constant $ \lambda $, by induction, for any $ p \geq 1 $, the iterate $ T^p $ is also a Banach contraction with constant $ \lambda^p $:
\[
d(T^p x, T^p y) \leq \lambda^p \, d(x, y).
\]
Because $ \lambda < 1 $, we have $ \lambda^p \to 0 $ as $ p \to \infty $. Therefore, there exists a minimal integer $ p_{\min} \geq 1 $ such that:
\[
\lambda^{p_{\min}} \leq \frac{1}{3}.
\]
It is a known result in fixed point theory, and we proved for both Chatterjea and Kannan (with constant $\leq \frac{1}{2}$) that any Banach contraction with constant $ \leq \frac{1}{3} $ is also a Chatterjea contraction. That is, for such $ p_{\min} $, we have:
\[
d(T^{p_{\min}} x, T^{p_{\min}} y) \leq \alpha \left[ d(x, T^{p_{\min}} y) + d(y, T^{p_{\min}} x) \right],
\]
with $ \alpha = \lambda^{p_{\min}} \leq \frac{1}{3} < \frac{1}{2} $. Hence, $ T $ is a Singh-Chatterjea contraction for $ p = p_{\min} $.
\end{proof}

\begin{remark}[Minimal $ p $]
The minimal such $ p $ is given explicitly by:
\[
p_{\min} = \left\lceil \frac{\ln(1/3)}{\ln \lambda} \right\rceil = \left\lceil \frac{\ln 3}{-\ln \lambda} \right\rceil,
\]
where $ \lceil \cdot \rceil $ denotes the ceiling function. This is the smallest integer $ p $ for which $ \lambda^p \leq \frac{1}{3} $.
\end{remark}

\begin{example}[Concrete case: $ \lambda = 0.9 $]
Consider the Banach contraction $ T: \mathbb{R} \to \mathbb{R} $ defined by $ T(x) = 0.9x $, with $ d(x,y) = |x - y| $. Here, $ \lambda = 0.9 $.

We compute:
\[
p_{\min} = \left\lceil \frac{\ln 3}{-\ln 0.9} \right\rceil = \left\lceil \frac{1.0986}{0.1054} \right\rceil = \left\lceil 10.43 \right\rceil = 11.
\]

Indeed:
\[
\lambda^{10} = 0.9^{10} \approx 0.3487 > \tfrac{1}{3} \approx 0.3333, \quad
\lambda^{11} = 0.9^{11} \approx 0.3138 < \tfrac{1}{3}.
\]

Thus, $ T^{11} $ is a Banach contraction with constant $ \approx 0.3138 \leq \frac{1}{3} $, hence is a Chatterjea contraction — so $ T $ is a Singh-Chatterjea contraction for $ p = 11 $.

Note that for $ p = 1 $, $ T $ is \emph{not} a Chatterjea contraction: for $ x = 1 $, $ y = 0 $, we have:
\[
|Tx - Ty| = 0.9, \quad |x - Ty| + |y - Tx| = |1 - 0| + |0 - 0.9| = 1.9, \quad \frac{0.9}{1.9} \approx 0.474 > \tfrac{1}{3},
\]
and in fact, no $ \alpha < \tfrac{1}{2} $ satisfies the Chatterjea condition for all $ x,y $ when $ \lambda = 0.9 $. But after 11 iterations, it does.
\end{example}

\begin{remark}[Universality of Singh-Chatterjea]
This result shows that the family of Singh-Chatterjea contractions (over all $ p \geq 1 $) is \emph{universal} for Banach contractions: 
\[
\text{Banach} \subset \bigcup_{p=1}^{\infty} \text{SinghChatterjea}_p.
\]
Moreover, since we have already exhibited mappings that are Singh-Chatterjea but not Banach (see Section~\ref{sec:scnotbanach}), the inclusion is strict:
\[
\text{Banach} \subsetneq \bigcup_{p=1}^{\infty} \text{SinghChatterjea}_p.
\]
This underscores the remarkable generality of the Singh-Chatterjea framework: it captures not only non-Banach mappings, but \emph{all} Banach mappings via iteration.
\end{remark}

 \begin{figure}[ht]
     \centering
     \includegraphics[width=0.8\textwidth]{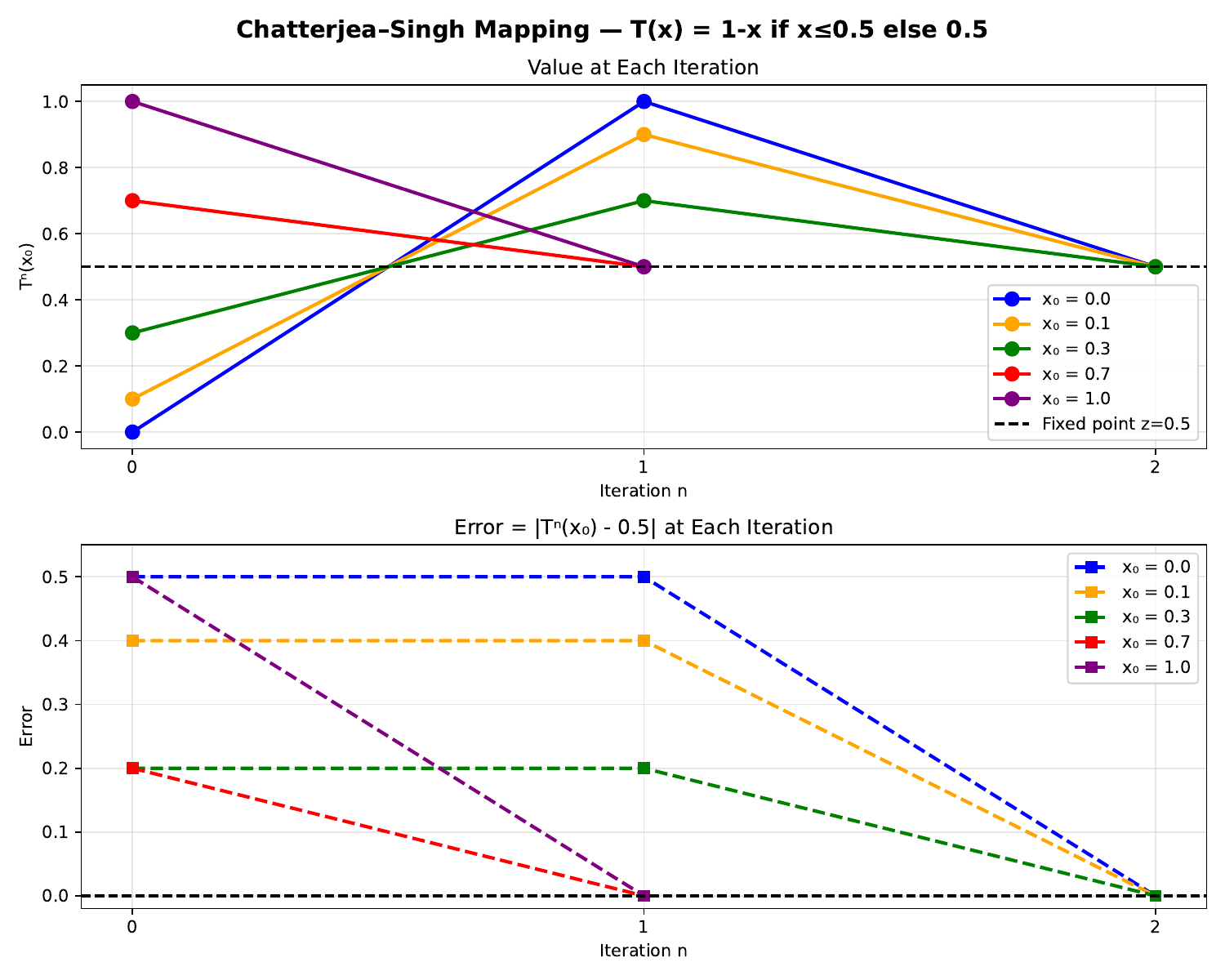}
     \caption{Convergence of a Singh-Chatterjea mapping (non-Banach example) from multiple initial points.}
     \label{fig:cs-convergence}
 \end{figure}

\section{Conclusion}
\label{sec:concl}
We have proposed a Singh-type generalization of the Chatterjea contraction, proved a fixed point theorem under this framework, and illustrated its validity through an example. Our results not only unify two classical contractions but also open a path for further applications in nonlinear analysis and numerical methods.


\end{document}